\documentclass[a4paper,12pt,reqno]{amsart}
\usepackage{tikz}
\usepackage{amsfonts}
\usepackage{amsmath}
\usepackage{amssymb}
\usepackage{mathrsfs}
\usepackage{lipsum}
\usepackage{newclude}
\usepackage{nomencl}
\makenomenclature
\usepackage{varioref}
\usepackage{hyperref}
\usepackage{cleveref}

\numberwithin{equation}{section}
\theoremstyle{plain}

\newtheorem{thm}{Theorem}[section]
\newtheorem{prop}[thm]{Proposition}

\theoremstyle{definition}
\newtheorem{defn}[thm]{Definition}

\newtheorem{quest}[thm]{Question}
\newcommand{\inlinemaketitle}{{\let\newpage\relax\maketitle}}
\newcommand{\be}{\begin{equation}}
\newcommand{\ee}{\end{equation}}

\def\TT{{\mathbb T}}
\def\NN{{\mathbb N}}
\def\dualSU2{\frac12\NN_0}

\def\ZZ{{\mathbb Z}}

\def\Gh{{\widehat{G}}}
\def\dpi{{d_\pi}}

\def\SU2{{\rm SU(2)}}

\DeclareMathOperator{\diag}{diag}

\DeclareMathOperator{\Ext}{Ext}
\DeclareMathOperator{\R}{Res}

\begin{document}
\title{Re-expansions on compact Lie groups}

\author[Rauan Akylzhanov]{Rauan Akylzhanov}
\address{
  Rauan Akylzhanov:
  \endgraf
 School of Mathematical Sciences
 \endgraf
Queen Mary University of London
\endgraf
United Kingdom
\endgraf
  {\it E-mail address} {\rm r.akylzhanov@qmul.ac.uk}
 }

\author[Elijah Liflyand]{Elijah Liflyand}
\address{
  Elijah Liflyand:
  \endgraf
  Department of Mathematics
  \endgraf
  Bar-Ilan University
  \endgraf
  Israel
  \endgraf
  {\it E-mail address} {\rm liflyand@gmail.com}
  }

\author[Michael Ruzhansky]{Michael Ruzhansky}
\address{
  Michael Ruzhansky:
  \endgraf
Department of Mathematics: Analysis,
Logic and Discrete Mathematics
  \endgraf
Ghent University, Belgium
  \endgraf
 and
  \endgraf
 School of Mathematical Sciences
 \endgraf
Queen Mary University of London
\endgraf
United Kingdom
\endgraf
  {\it E-mail address} {\rm michael.ruzhansky@ugent.be}
 }


\date{\today}

\keywords{Fourier series, re-expansion, compact Lie groups, Hilbert transform}

\thanks{The authors were supported in parts by the FWO Odysseus Project, EPSRC grant EP/R003025/1 and by the Leverhulme Grant RPG-2017-151.}
\begin{abstract}
In this paper we consider the re-expansion problems on compact Lie groups. First, we establish weighted versions of classical re-expansion results in the setting of multi-dimensional tori. 
A natural extension of the classical re-expansion problem to general compact Lie groups can be formulated as follows: given a function on the maximal torus of a compact Lie group, what conditions on its (toroidal) Fourier coefficients   
are sufficient in order to have that the group Fourier coefficients of its central extension are summable. We derive the necessary and sufficient conditions for the above property to hold in terms of the root system of the group.
Consequently, we show how this problem leads to the re-expansions of even/odd functions on compact Lie groups, giving a necessary and sufficient condition in terms of the discrete Hilbert transform and the root system.
In the model case of the group $\SU2$ a simple sufficient condition is given.
\end{abstract}
\maketitle

\section{Introduction}

In the 50-s (see, e.g., \cite{IT1955} or in more detail \cite[Chapters II
and VI]{Kah}), the following problem in Fourier Analysis attracted
much attention:

{\it Let $\{a_k\}_{k=0}^\infty$ be the sequence of the Fourier
coefficients of the absolutely convergent sine (cosine) Fourier
series of a function $f:\mathbb T=[-\pi,\pi)\to \mathbb C,$ that is
$\sum |a_k|<\infty.$ Under which conditions on $\{a_k\}$ the
re-expansion of $f(t)$ ($f(t)-f(0)$, respectively) in the cosine
(sine) Fourier series will also be absolutely convergent?}

The obtained sufficient condition, sharp on the whole class, is quite simple and is the same in both
cases:

\begin{equation}
\label{condser-0}
\sum\limits_{k=1}^\infty |a_k|\ln(k+1)<\infty.
 \end{equation}

In \cite{L0}, a similar problem of the integrability of the
re-expansion for Fourier transforms of functions defined on $\mathbb R_+=[0,\infty)$
has been studied. Surprisingly, necessary and sufficient conditions in terms of the membership of
the sine or cosine Fourier transform in a certain Hardy space have been found.

In the present paper, we consider a similar problem on compact Lie groups.
There are special features in this study. A natural analogue of the classical re-expansion problem on a general compact Lie group $G$ can be formulated as follows: given a function $f=\sum\limits_{k\in\ZZ^l}a_k e^{ikt}$ on the maximal torus of $G$, what conditions on its Fourier coefficients $\{a_k\}_{k\in\ZZ^l}$  
are sufficient in order to have that the group Fourier coefficients of its central extension are summable (namely, to have $\widehat{\Ext[f]}\in\ell^1_{sch}(\Gh)$, with the appearing symbols explained in the following sections). The necessary and sufficient condition for this will be given in Theorem \ref{THM:Q1} in terms of the root structure of the group. Consequently, in Theorem \ref{THM:Q2} we show how this problem leads to the re-expansions of even/odd functions on compact Lie groups, and we derive a necessary and sufficient condition for it. In Theorem \ref{THM:Q2-SU2} we demonstrate the obtained criterion on the case of the compact Lie group $\SU2$.

The outline of the paper is as follows. In the next section we present general results
for compact Lie groups $G$ and a special case $G=\SU2$. Then we give detailed proofs.
In Section \ref{ac} we improve those old known results by obtaining necessary and sufficient
conditions rather than just \eqref{condser-0}. 

\section{Absolutely convergent Fourier series}\label{ac}

In this section we refine (\ref{condser-0}) in the sense that, like for the Fourier transforms,
the necessary and sufficient conditions will be established. Analogously to
the case of Fourier transform, they will be given in terms of the integrability
of Hilbert transforms, but discrete. The background for this analysis will be given in the next subsection.
After proving one-dimensional results, we will give their multivariate extensions.
However, for generalisations to compact Lie groups it is more representative to deal with the
weighted absolute convergence. We will present corresponding estimates both in
dimension one and in several dimensions.

\subsection{Discrete Hilbert transforms}

For the sequence $a=\{a_k\}\in\ell^1,$ $\ell^1:=\ell^1(\mathbb Z_+)$, the discrete Hilbert
transform is defined for $n\in\mathbb Z$ as (see, e.g., \cite[(13.127)]{King})

\begin{eqnarray}\label{dht} \hbar a(n)=\sum\limits_{k=-\infty
\atop k\ne n}^\infty\frac{a_k}{n-k}. \end{eqnarray}

If the sequence $a$ is either even or odd, the corresponding Hilbert
transforms $\hbar^e$ and $\hbar^o$ may be expressed in a special
form (see, e.g., \cite{and77} or \cite[(13.130) and
(13.131)]{King}). More precisely, if $a$ is even, with $a_0=0,$ we
have $\hbar^e(0)=0$ and we define

\begin{eqnarray}
\label{dhte}
 \hbar^e a(n)
 =
 \sum\limits_{k=1
\atop k\ne n}^\infty\frac{2na_k}{n^2-k^2}+\frac{a_n}{2n},\quad n=1,2,\ldots.
\end{eqnarray}
If $a$ is odd, with $a_0=0,$ we define 

\begin{eqnarray}\label{dhto} \hbar^o a(n)=\sum\limits_{k=1
\atop k\ne n}^\infty\frac{2ka_k}{n^2-k^2}-\frac{a_n}{2n}, ,\quad n=0,1,2,\ldots.
\end{eqnarray}
Of course, $\frac{a_0}{0}$ is considered to be zero.

One of the best sources for the theory of discrete Hilbert transforms
and Hardy spaces is \cite{BoCa}. For weighted estimates for them, see
\cite{and77}, \cite{ST1}, \cite{ST2} and \cite{L1}.

\subsection{One-dimensional case}

The obtained condition is quite simple and is the same in both cases:

\begin{eqnarray}\label{condser}\sum\limits_{k=1}^\infty|a_k|\ln(k+1)<\infty. \end{eqnarray}

Analysing the proof, say, in \cite{IT1955}, one can see that in fact more
general results are hidden in the proofs. They can be given in terms of
the (discrete) Hilbert transform.

\begin{thm}
\label{ans}
In order than the re-expansion $\sum b_k \sin kt$
of $f(t)-f(0)$ with the absolutely convergent cosine Fourier series be
absolutely convergent, it is necessary and sufficient that the discrete Hilbert
transform ${\hbar}a$ of the sequence $a$ of the sine Fourier coefficients of $f$ is summable.

Similarly, in order than the re-expansion $\sum b_k\cos kt$ of $f$ with the
absolutely convergent sine Fourier series be absolutely convergent, it is
necessary and sufficient that the discrete Hilbert transform ${\hbar}a$
of the sequence $a$ of the cosine Fourier coefficients of $f$ is summable.
\end{thm}

What, in fact, is proven in the mentioned papers, for $b=\{b_k\}_{k=0}^\infty,$

\begin{eqnarray}\label{cosi}b={\hbar}^e_- a\end{eqnarray}
in the first part of Theorem \ref{ans}, and

\begin{eqnarray}\label{sico}b={\hbar}^o_- a\end{eqnarray}
in the second one, where

\begin{eqnarray}\label{cosih}b={\hbar}^e_- a=\frac2\pi\sum\limits_{k=1\atop k-n\, odd}^\infty
a_k\left(\frac1{n+k}+\frac1{n-k}\right)=\frac4\pi\sum\limits_{k=1\atop k-n\, odd}^\infty
\frac{na_k}{n^2-k^2}\end{eqnarray}
and

\begin{eqnarray}\label{sicoh}b={\hbar}^o_- a=\frac2\pi\sum\limits_{k=1\atop k-n\, odd}^\infty
a_k\left(\frac1{n+k}+\frac1{k-n}\right)=\frac4\pi\sum\limits_{k=1\atop k-n\, odd}^\infty
\frac{ka_k}{k^2-n^2}\end{eqnarray}
is the halved even and odd discrete Hilbert transform, respectively. Indeed, for $f:[0,\pi]\to\mathbb C$,
with $f(0)=0$ for simplicity (which is necessary as well as $f(\pi)=0$), in the first case

\begin{eqnarray}\label{se} f_e(t)=\sum\limits_{k=1}^\infty a_k\cos kt,\end{eqnarray}
with $a\in\ell^1$, while

\begin{eqnarray}\label{fo}f_s(t)=\sum\limits_{n=1}^\infty b_n\sin nt,\end{eqnarray}
with

\begin{align*} b_n&=\frac2\pi\int_0^\pi f(t)\sin nt\,dt=\frac2\pi\sum\limits_{k=1}^\infty a_k\int_0^\pi \cos kt\sin nt\,dt\\
&=\frac1\pi\sum\limits_{k=1}^\infty a_k\int_0^\pi [\sin(n+k)t+\sin (n-k)t]\,dt\\&=\frac2\pi\sum\limits_{k=1\atop k-n\, odd}^\infty
a_k\left(\frac1{n+k}+\frac1{n-k}\right).\end{align*}
Similarly, in the second case

\begin{eqnarray}\label{so} f_s(t)=\sum\limits_{k=1}^\infty a_k\sin kt,\end{eqnarray}
with $a\in\ell^1$, while

\begin{eqnarray}\label{fe} f_e(t)=\sum\limits_{n=1}^\infty b_n\cos nt,\end{eqnarray}
with

\begin{align*} b_n&=\frac2\pi\int_0^\pi f(t)\cos nt\,dt=\frac2\pi\sum\limits_{k=1}^\infty a_k\int_0^\pi \cos nt\sin kt\,dt\\
&=\frac1\pi\sum\limits_{k=1}^\infty a_k\int_0^\pi [\sin(n+k)t+\sin (k-n)t]\,dt\\&=\frac2\pi\sum\limits_{k=1\atop k-n\, odd}^\infty
a_k\left(\frac1{n+k}+\frac1{k-n}\right).\end{align*}
This can be continued as follows:

\begin{align}\label{halfodht} \sum\limits_{n=1}^\infty\left|\sum\limits_{k=1\atop k-n\, odd}^\infty
a_k\left(\frac1{n+k}+\frac1{k-n}\right)\right|&=\sum\limits_{n=1}^\infty\biggl|\sum\limits_{k=1\atop k-n\, odd}^\infty
a_k\left(\frac1{n+k}+\frac1{k-n}\right)\nonumber\\
&+\frac12\sum\limits_{k=1\atop k-n\, odd, k\ne n+1}^\infty a_k\left(\frac1{n+1+k}+\frac1{k-n-1}\right)\nonumber\\
&-\frac12\sum\limits_{k=1\atop k-n\, odd, k\ne n+1}^\infty a_k\left(\frac1{n+1+k}+\frac1{k-n-1}\right)\biggr|\nonumber\\
=\frac12\sum\limits_{n=1}^\infty\biggl|\sum\limits_{k=1\atop k\ne n}^\infty a_k\left(\frac1{n+k}+\frac1{k-n}\right)
&+\sum\limits_{k=1\atop k-n\, odd}^\infty a_k\left(\frac1{n+k}-\frac1{k+n+1}\right)\nonumber\\&+
\sum\limits_{k=1\atop k-n\, odd}^\infty a_k\left(\frac1{k-n}-\frac1{k-n-1}\right)\biggr|.\end{align}
Since

$$\sum\limits_{n=1}^\infty\left|\sum\limits_{k=1\atop k-n\, odd}^\infty a_k\left(\frac1{k-n}-\frac1{k-n-1}\right)\right|
\le C\sum\limits_{k=1}^\infty |a_k|,$$
and the same is for the preceding sum, we have

\begin{eqnarray}\label{equivo}\|{\hbar}^o_- a\|_{\ell^1}=\frac12\|{\hbar}^o a\|_{\ell^1}+O(1).\end{eqnarray}
In the completely similar way one can prove that

\begin{eqnarray}\label{equive}\|{\hbar}^e_- a\|_{\ell^1}=\frac12\|{\hbar}^e a\|_{\ell^1}+O(1).\end{eqnarray}
Hence, replacing the necessary and sufficient condition of the summability of the halved Hilbert transforms,
which follows immediately from (\ref{cosi}) and (\ref{sico}), with the summability of the usual discrete
Hilbert transform, one arrives at the proof of the theorem.

In this case (\ref{condser}) is just a sufficient condition for the
summability of the discrete Hilbert transform, though sharp on the whole class.

\subsection{Multidimensional case}

Let $\eta=(\eta_1, ...,\eta_d)$ denote a $d$-dimensional vector with $\eta_j=0$ or
$\eta_j=1$ only, and $|\eta|=\eta_1+...+\eta_d\le d$. Of course, the vectors $\bf0=(0,0,...,0)$ and $\bf1=(1,1,...,1)$ are
among such vectors. The vectors $\chi$ and $\zeta$ will be understood and used similarly.

Starting from a function $f: \mathbb T^d_+=[0,\pi]^d\to \mathbb C$, with $f(x_1,...,x_{j-1},0,x_{j+1},...,x_d)=0,$
$j=1,2,...,d,$ we consider

\begin{equation}
\label{EQ:eta-symm}
f_\eta(x)=\sum\limits_{k\in\mathbb Z_+} a_k (\prod\limits_{j:\eta_j=1}\cos k_jx_j)
(\prod\limits_{j:\eta_j=0}\sin k_jx_j),
\end{equation}
with $a\in\ell^1$, where now $\ell^1:=\ell^1(\mathbb Z_+^d)$. The problem is under what conditions
in the re-expansion

$$f_{\bf1-\eta}(x)=\sum\limits_{m\in\mathbb Z_+} b_m (\prod\limits_{j:\eta_j=0}\cos m_jx_j)
(\prod\limits_{j:\eta_j=1}\sin m_jx_j)$$
we have $b\in \ell^1$. Since

\begin{align}\label{coefd} b_m&=\frac{2^d}{\pi^d}\int_{\mathbb T^d_+} f(x)(\prod\limits_{j:\eta_j=0}\cos m_jx_j)
(\prod\limits_{j:\eta_j=1}\sin m_jx_j)\,dx\\
&=\frac{2^d}{\pi^d}\sum\limits_{k\in\mathbb Z_+} a_k\int_{\mathbb T^d_+}(\prod\limits_{j:\eta_j=1}\cos k_jx_j\sin m_jx_j)
(\prod\limits_{j:\eta_j=0}\sin k_jx_j\cos m_jx_j)\,dx\nonumber\\
&=\frac{2^d}{\pi^d}\sum\limits_{k\in\mathbb Z_+\atop k_j-m_j\, odd, j=1,2,...,d}a_k\,\prod\limits_{j:\eta_j=1}
\left(\frac1{m_j+k_j}+\frac1{m_j-k_j}\right)\,\prod\limits_{j:\eta_j=0}\left(\frac1{m_j+k_j}+\frac1{k_j-m_j}\right).
\nonumber  \end{align}
It is quite natural to denote the right-hand side of (\ref{coefd}) by ${\hbar}_\eta^- a$
\begin{align}
\label{EQ:HT-def}
&{\hbar}_\eta^-a\\
=\sum\limits_{k\in\mathbb Z_+\atop k_j-m_j\, odd, j=1,2,...,d} & a_k\,\prod\limits_{j:\eta_j=1}
\left(\frac1{m_j+k_j}+\frac1{m_j-k_j}\right)\,\prod\limits_{j:\eta_j=0}\left(\frac1{m_j+k_j}+\frac1{k_j-m_j}\right).
\end{align}
The summability of ${\hbar}_\eta^-a$
is the necessary and sufficient condition for $b\in \ell^1$. This is a direct multidimensional generalisation of
the one-dimensional necessary and sufficient conditions of the summability of (\ref{cosi}) and (\ref{sico}).

Further, applying (\ref{halfodht}) in each variable, we arrive at an analog of (\ref{equivo}) and (\ref{equive}):

\begin{eqnarray}\label{equivd} \|{\hbar}_\eta^- a\|_{\ell^1}=\frac12\|{\hbar}_\eta^e {\hbar}_{\bf1-\eta}^o a\|_{\ell^1}+
O\left(\sum\limits_{0\le|\chi|+|\zeta|<|\eta|}\|{\hbar}_\chi^e {\hbar}_\zeta^o a\|_{\ell^1}\right),\end{eqnarray}
where

$${\hbar}_\chi^e a=(\prod\limits_{j:\chi_j=1} {\hbar}_j^e)a, \qquad {\hbar}_\zeta^o a=(\prod\limits_{j:\zeta_j=1} {\hbar}_j^o)a,$$
with ${\hbar}_j^e$ and ${\hbar}_j^o$ being ${\hbar}^e$ and ${\hbar}^o$, respectively, applied to the $j$th component of $a$.
In other words, the operators on the right-hand side of (\ref{equivd}) are mixed discrete Hilbert transforms. Their integrability
leads, correspondingly, to the hybrid discrete Hardy spaces (see \cite{L2}).

Thus, the right-hand side of (\ref{equivd}) consists of the leading term (repeated discrete Hilbert transforms applied to EACH of
the components of $a$) and the remainder term that consists of the $\ell^1$ norms of the repeated Hilbert transforms applied
only to a proper part of the components of $a$. Of course, $\|a\|_{\ell^1}$ is among them.

Therefore, $\|{\hbar}_\eta^e {\hbar}_{\bf1-\eta}^o a\|_{\ell^1}<\infty$ is a necessary and sufficient condition for $b\in \ell^1$
provided the remainder term in (\ref{equivd}) is finite.

And, of course, applying (\ref{condser}) in each variable, we have a sufficient condition

\begin{eqnarray}
\label{condserd}
\sum\limits_{k\in\mathbb Z_+^d}|a_k|\prod\limits_{j=1}^d\ln(k_j+1)<\infty.
\end{eqnarray}

\subsection{One-dimensional weighted case}

A more general problem arises if one assumes $\{k^q a_k\}$, $q=1,2,...,$ belongs to $\ell^1$
and figures out when  $\{n^q b_n\}$ belongs to $\ell^1$.
In fact, this means that if we start with (\ref{se}), the question reduces to that about

\begin{eqnarray}\label{seq} f_e^{(q)}(t)=\sum\limits_{k=1}^\infty k^q a_k\cos (kt+\frac{q\pi}2),\end{eqnarray}
and

\begin{eqnarray}\label{foq} (f_e^{(q)})_s(t)=\sum\limits_{n=1}^\infty n^q b_n\sin (nt+\frac{q\pi}2),\end{eqnarray}
while if we start with (\ref{so}), the question reduces to that about

\begin{eqnarray}\label{soq} f^{(q)}_s(t)=\sum\limits_{k=1}^\infty k^q a_k\sin(kt+\frac{q\pi}2),\end{eqnarray}
and

\begin{eqnarray}\label{feq} (f_s^{(q)})_e(t)=\sum\limits_{n=1}^\infty n^q b_n\cos (nt+\frac{q\pi}2).\end{eqnarray}
The problem, in fact, reduces to the initial one if one observes that

\begin{eqnarray}\label{conves} (f_e^{(q)})_s(t)=\pm f_s^{(q)}(t)\end{eqnarray}
and
\begin{eqnarray}\label{convse} (f_s^{(q)})_e(t)=\pm f_e^{(q)}(t).\end{eqnarray}
This follows from integration by parts $q$ times, where the integrated terms vanish sometimes automatically
or by assuming $f^{(j)})_s(0)$ and $f^{(j)}(\pi)=0$, $j=0,1,...,q-1.$
Taking now into account the arguments of the previous section leads to the following assertion.
To present it, denote $a^q=\{A_k^q\}=\{k^q a_k\}$ and $b^q=\{B_k^q\}=\{k^q b_k\}$. Since we study the $\ell^1$ summability
of these sequences, no matter if $-a^q$ is taken instead, or similarly $-b^q$.
Also, since $q$ is integer, $\cos (kt+\frac{q\pi}2)$ is either $\cos kt$ or $\sin kt$,
while $\sin (kt+\frac{q\pi}2)$ is, correspondingly $\sin kt$ or $\cos kt$, each time
up to a sign $\pm$.

\begin{thm}\label{ansq} Let $f^{(j)}(0)=f^{(j)}(\pi)=0$, $j=0,1,...,q-1.$
In order than the re-expansion $\sum B_k \sin kt$
of $f^{(q)}(t)$ with the absolutely convergent cosine Fourier series with coefficients
$a^q$ be absolutely convergent, it is necessary and sufficient that the discrete Hilbert
transform ${\hbar}^e a^q$ of the sequence $a^q$ is summable.

Similarly, in order than the re-expansion $\sum B_k\cos kt$ of $f^{(q)}$ with the
absolutely convergent sine Fourier series with coefficients $a^q$ be absolutely convergent, it is
necessary and sufficient that the discrete Hilbert transform ${\hbar}^o a^q$
of the sequence $a^q$ is summable.
\end{thm}

And in both cases the (sharp) sufficient condition is

\begin{eqnarray}\label{condserq}\sum\limits_{k=1}^\infty k^q|a_k|\ln(k+1)<\infty. \end{eqnarray}

\subsection{Multidimensional weighted case}

We will generalise the results from the previous subsection to several dimensions.
Let now $q=(q_1,...,q_d)$ be a vector with integer $q_j\ge0$. For $k\in\mathbb Z^d_+$,

$$k^q=k_1^{q_1}...k_d^{q_d}.$$
Our starting assumption will now be $k^q a_k\in\ell^1(\mathbb Z_+^d)$.

In what follows $D^q f$ will mean the partial derivative

\begin{eqnarray*} D^q f(x)=\left(\prod\limits_{j=1}^d
\frac{{\partial}^{q_j}}{\partial x_j^{q_j}}\right)f(x).          \end{eqnarray*}

Starting from a function $f: \mathbb T^d_+=[0,\pi]^d\to \mathbb C$, with $f(x_1,...,x_{j-1},0,x_{j+1},...,x_d)=0,$
$j=1,2,...,d,$ we consider

$$D^qf_\eta(x)=\sum\limits_{k\in\mathbb Z_+} k^q a_k (\prod\limits_{j:\eta_j=1}\cos (k_jx_j+\frac{q_j\pi}2))
(\prod\limits_{j:\eta_j=0}\sin (k_jx_j+\frac{q_j\pi}2)),$$
with $k^q a\in\ell^1$, where now $\ell^1:=\ell^1(\mathbb Z_+^d)$. The problem is under what conditions
in the re-expansion of $f$ in the form

$$(D^qf_\eta)_{\bf1-\eta}(x)=\sum\limits_{m\in\mathbb Z_+} m^q b_m (\prod\limits_{j:\eta_j=0}\cos (m_jx_j+\frac{q_j\pi}2))
(\prod\limits_{j:\eta_j=1}\sin (m_jx_j+\frac{q_j\pi}2))$$
we have $m^q b\in \ell^1$.

The proof of the next result is just a superposition and combination of the arguments from
the two previous sections. As above, we can think on $\cos k_jx_j$ and $\sin k_jx_j$ instead of
$\cos (k_jx_j+\frac{q_j\pi}2))$ and $\sin (k_jx_j+\frac{q_j\pi}2))$.

\begin{thm}\label{ansqd} Let $D^sf(x_1,...x_{j-1},0,x_{j+1},...,x_d)=D^sf(x_1,...x_{j-1},\pi,x_{j+1},...,x_d)=0$,
$j=0,1,...,q_j-1$, for any $s=(s_1,...,s_d)$ with $s_j=0,1,...,q_j-1$.
In order that the re-expansion

$$\sum\limits_{m\in\mathbb Z^d} m^q b_m (\prod\limits_{j:\eta_j=0}\cos m_jx_j)
(\prod\limits_{j:\eta_j=1}\sin m_jx_j)$$
of $D^qf$ with the absolutely convergent Fourier series with coefficients
$k^q a$ is absolutely convergent, it is necessary and sufficient that the discrete Hilbert
transform ${\hbar}_\eta^e {\hbar}_{\bf1-\eta}^o$ of the sequence $k^q a$ is summable
provided

$$\sum\limits_{0\le|\chi|+|\zeta|<|\eta|}\|{\hbar}_\chi^e {\hbar}_\zeta^o k^q a\|_{\ell^1}$$
is finite.
\end{thm}

And, of course, applying (\ref{condserq}) in each variable, we have a sufficient condition

\begin{eqnarray}\label{condserdq}\sum\limits_{k\in\mathbb Z_+^d} k^q|a_k|\prod\limits_{j=1}^d\ln(k_j+1)<\infty. \end{eqnarray}

\section{Main results}
Let $G$ be a compact connected simply connected Lie group of rank $l$ and denote by $\Gh$ its unitary dual, i.e. the set of all irreducible inequivalent representations $\pi$.
There is a one-to-one correspondence between $\Gh$ and the set $\Lambda$ of the highest weights
\begin{equation}
\Gh\ni\pi \leftrightarrow \mu=(\mu_1,\ldots,\mu_l)\in\Lambda.
\end{equation}
Therefore, we use the symbol $\pi$ to denote both the class of equivalent representations and the corresponding highest weight $\mu$, i.e.
\begin{equation}
\Gh\ni\pi=\{\xi\in\Gh \colon \xi \,\text{ is equivalent to } \,\pi\} = (\pi_1,\ldots,\pi_l)\in\Lambda,
\end{equation}
where $\pi_k=\mu_k,\,k=1,\ldots,l$.
Each irreducible representation $\pi\in\Gh$ has the corresponding highest weight $\mu=(\mu_1,\ldots,\mu_l)$.
The Killing form restricted to $\mathfrak{t}$ allows us to define isometric isomorphism between the representation weights and the elements of $\mathfrak{t}$.

Denote by $T$ the maximal abelian subgroup of $G$. We identify $T$ with the $l$-dimensional torus $\TT^l$ where $l$ is the rank of $G$.
Let $R_+$ be a system of positive roots and let $W$ denote the Weyl group.
If we denote by $\chi_{\pi}$ the character corresponding to $\pi\in\Gh$ we have the Weyl character formula
\begin{equation}
\label{EQ:Weyl}
\chi_{\pi}(e^{iH})
=
\frac
{
\sum\limits_{w\in W}\det(w)e^{i(w\cdot(\mu+\delta),H)}
}
{
\sum\limits_{w\in W}\det(w)e^{i(w\cdot\delta,H)}
}
,\quad H\in\mathfrak{t},
\end{equation}
where $$\Delta(e^H)=\sum\limits_{w\in W}\det(w)e^{i(w\cdot\delta,H)}$$ is the Weyl function and $\delta$ denotes the half-sum of all positive roots.


We say that a function $g$ on $G$ is {\it central} if it satisfies
\begin{equation}
g(yty^{-1})=g(t)
\end{equation}
for all $y,t\in G$.
There is also the Weyl integral formula for central functions on $G$,
\begin{equation}
\label{EQ:Weyl-integral}
\int_{G}f(g)\,dg
=
\frac1{|W|}\int_{\TT^l}f(t)|\Delta(t)|^2\,dt,
\end{equation}
where $|W|$ denotes the cardinality of the Weyl group $W$(i.e. the number of elements in $W$).
It can be easily shown that
\begin{equation}
\label{EQ:Weyl-denominator0}
\Delta(t)
=
\prod\limits_{\alpha\in R_+}
(e^{i(\alpha,H)}+e^{-i(\alpha,H)}-2),
\end{equation}
where $R_+$ are the positive roots associated with $G$.

Let $a=\{a_k\}_{k\in\ZZ^l}\in\ell^1(\ZZ^l)$ and let
\begin{equation}
\label{EQ:torus-expansion}
f
=
\sum\limits_{k\in\ZZ^l}a_k e^{ikt}.
\end{equation}
Let us denote by $\Ext[f]$ the central extension of $f$ to the group $G$, i.e.
\begin{equation*}
\{\text{functions on $\TT^l$}\}\ni f \mapsto \Ext[f]\in\{\text{functions on $G$}\},
\end{equation*}
where we define $\Ext[f]$ as follows
\begin{equation*}
\Ext[f](y^{-1}ty):=f(t),t\in\TT^l,y\in G.
\end{equation*}
It is clear that $\Ext[f]$ is a central function.
Analogously, we denote by $\R$ the restriction of $f$ from $G$ to $\TT^l$
\begin{equation*}
\{\text{functions on $G$}\}\ni f \mapsto \R[f] \in \{\text{functions on $\TT^l$}\},
\end{equation*}
where we define $\R[f]$ as follows
\begin{equation*}
\R[f]:=f\big|_{\TT^l}.
\end{equation*}
Motivated by the classical case, we seek summability conditions on the group Fourier coefficients in terms of the $\ell^p$-Schatten spaces $\ell^p_{sch}(\Gh)$ on $\Gh$.
Thus, for $1\leq p<\infty$, we define the space $\ell^p_{sch}(\Gh)$ of matrix valued sequences $\{\sigma=\{\sigma(\pi)\}_{\pi\in\Gh}\}$ endowed with the norm
$$
\|\sigma\|_{\ell^p_{sch}(\Gh)}
=
\left(
\sum\limits_{\pi\in\Gh}
\dpi
\|\sigma(\pi)\|^p_{S^p}
\right)^{\frac1p}
$$
with the obvious modification for $p=\infty$, where $S^p$ are the usual Schatten-von Neumann norms of matrices.

For a function $f\in L^1(G)$, as usual, we denote its Fourier coefficients by
$$
\widehat{f}(\pi)=\int_G f(g) \pi(g)^* dg,
$$
where $dg$ is the bi-invariant Haar measure on $G$. We refer to \cite{Ruzhansky:2010aa} for the necessary backgrounds for the Fourier analysis on the compact Lie groups.

\begin{quest}
Given a function $f=\sum\limits_{k\in\ZZ^l}a_k e^{ikt}$, what conditions on its Fourier coefficients $\{a_k\}_{k\in\ZZ^l}$  
are sufficient in order to have $\widehat{\Ext[f]}\in\ell^1_{sch}(\Gh)$.
\end{quest}
\begin{thm}
\label{THM:Q1}
Suppose that
$$
f=\sum\limits_{k\in\ZZ^l}a_k e^{ikt}.
$$
Then the Fourier coefficients $\widehat{\Ext[f]}$ of its extension $\Ext[f]$ belong to $\ell^1_{sch}(\Gh)$ if and only if
\begin{equation}
\label{EQ:condition}
\sum\limits_{\pi\in\Gh}
\dpi
\sum\limits^{\dpi}_{m=1}
\left|
\sum\limits^v_{j=1}a_{j}\widehat{\Delta^2}(\pi_{mm}-j)
\right|
<+\infty,
\end{equation}
where the number $0<v \leq \left|W\right|$ depend only on the Weyl group $W$ and
$d_{\pi}$ denotes the dimension of the representation $\pi\in\Gh$
\end{thm}

The number $v$ in \eqref{EQ:condition} appears as follows: in view of the formula \eqref{EQ:Weyl-denominator0}, only a finite number of the (toroidal) Fourier coefficients of $\Delta^2$ are non-zero, leading to the finite number $v$ in \eqref{EQ:condition}.

The explicit expression of $d_{\pi}$ is given by  Weyl's dimension formula in \eqref{EQ:Weyl}.

\begin{prop}
\label{PROP:sa-even}
Let $f\in L^1(G)$ and $\pi\in\Gh$. Then we have
\begin{equation}
\widehat{f}(\pi)=\widehat{f}(\pi)^*,
\end{equation}
if and only if
\begin{eqnarray}
f \text{is real-valued},\\
f(g^{-1})=f(g),\quad g\in G.
\end{eqnarray}
\end{prop}
\begin{defn}
\label{DEF:even-odd-G}
 A real-valued function $f$ on $G$ is called even or odd if
\begin{eqnarray}
f(g)=f(g^{-1}),\\
\text{or}\nonumber\\
f(g)=-f(g^{-1}),
\end{eqnarray}
respectively. 
\end{defn}
\begin{prop}
\label{PROP:odd-even}
  A function $f$ is odd or even on $G=U(n)$ in the sense of Definition \ref{DEF:even-odd-G} if and only if
  its restriction $f\big|_{\TT^l}$ to the maximal torus $\TT^l$ is odd or even respectively.
\end{prop}
Here $n=l$.
\begin{proof}[Proof of Proposition \ref{PROP:odd-even}]
Let $f$  be even. The case of odd function is analogous. Denote $h=f\big|_{\TT^l}$.  We write

\begin{eqnarray*}
e^{2\pi i t}
=
\left[
\begin{array}{ccc}
 e^{2\pi i t_1} & 0  & 0  \\
0  & \ddots  &  0 \\
0  & 0  &   e^{2\pi i t_n}
\end{array}
\right].
\end{eqnarray*}
Since $f$ is central and odd (see Definition \ref{DEF:even-odd-G}), we have
$$
h(t)=
f(e^{2\pi it})
=
f(e^{-2\pi it})
=
h(-t).
$$
If $h$ is odd, then we have
\begin{multline*}
f(x)=
f(y^{-1}e^{2\pi i t}y)
=
f(e^{2\pi i t})
=
h(t)
=
h(-t)
\\=
f(e^{-2\pi i t})
=
f(y^{-1}e^{-2\pi i t}y)
=
f(x^{-1}), \quad x,y\in G.
\end{multline*}
This completes the proof.
\end{proof}
We write $g_{\eta}=\R[f]$ and $\tilde{f}=\Ext[g_{1-\eta}]$.
The combination of Theorem \ref{THM:Q1} and Theorem \ref{ans} immediately yields the following theorem.
\begin{thm}
\label{THM:Q2}
Let $f\in L^1(G)$. Suppose that $f$ is even and its Fourier coefficients $\widehat{f}$ are integrable over $\Gh$:
\begin{equation}
\label{EQ:f-hat-integrable}
\widehat{f}\in\ell^1_{sch}(\Gh).
\end{equation}
Then the Fourier coefficients $\widehat{\Ext[f_{odd}]}$ of its odd re-expansion $f_{odd}$ are integrable
\begin{equation}
\widehat{\Ext[f_{odd}]}\in\ell^1_{sch}(\Gh),
\end{equation}
if and only if
\begin{equation}
\sum\limits_{\pi\in\Gh}
\dpi
\sum\limits^{\dpi}_{m=1}
\left|
{\hbar}
\left(
\sum\limits^v_{j=1}a_{j}\widehat{\Delta^2}(\pi_{mm}-j)
\right)
\right|
\leq
C
\sum\limits_{\pi\in\Gh}
\dpi
\sum\limits^{\dpi}_{m=1}
\left|
\sum\limits^v_{j=1}a_{j}\widehat{\Delta^2}(\pi_{mm}-j)
\right|,
\end{equation}
with the notations of Theorem \ref{THM:Q1}.
\end{thm}
For the model case $G=\SU2$, we give a simple sufficient condition. Here we use the standard identification $\widehat{\SU2}\simeq \frac12\mathbb N_0$, see e.g. \cite{Vilenkin:1968aa} or \cite{Ruzhansky:2010aa}.

\begin{thm}
\label{THM:Q2-SU2}
 Let $f\in L^1(\SU2)$. Suppose that $f$ is even and its Fourier coefficients $\widehat{f}$ are integrable over $\widehat{\SU2}$:
\begin{equation}
\widehat{f}\in\ell^1_{sch}(\widehat{\SU2}).
\end{equation}
Then the Fourier coefficients $\widehat{\Ext[f_{odd}]}$ of its odd re-expansion $f_{odd}$ are integrable
\begin{equation}
\widehat{\Ext[f_{odd}]}\in\ell^1_{sch}(\widehat{\SU2}),
\end{equation}
provided that
\begin{equation}
\sum\limits_{l\in\frac12\NN_0}
(2l+1)
\log(2l+1)
|a_{2l+1}|
<+\infty.
\end{equation}
\end{thm}

\section{Proofs}
\begin{proof}[Proof of Theorem \ref{THM:Q2-SU2}]
It can be checked straightforwardly
\begin{equation*}
\sum\limits^{2l+1}_{m=1}
\left(
a_{m-2}-2a_m+a_{m+2}
\right)
=
a_{-1}-a_1-a_{2l}
-a_{2l+1}
+a_{2l+2}
+a_{2l+3},
\end{equation*}
where without the loss of generality we can assume that $a_0=a_{-1}=0$.
Thus, we get
\begin{equation}
\sum\limits^{2l+1}_{m=1}
a_{m-2}-2a_m+a_{m+2}
=
-a_{2l}
-a_{2l+1}
+a_{2l+2}
+a_{2l+3}.
\end{equation}
Then the series
\begin{equation*}
\sum\limits_{l\in\frac12\NN_0}(2l+1)
\left|
{\hbar}
\sum\limits^{2l+1}_{m=1}
a_{m-2}-2a_m+a_{m+2}
\right|
\end{equation*}
is convergent if the series
\begin{equation}
\label{EQ:convergence}
\sum\limits_{l\in\frac12\NN_0}(2l+1)
\left|
{\hbar}
a_{2l+1}
\right|
\end{equation}
is convergent.
Let us denote $n=2l+1,l\in\frac12\NN_0$. Then we rewrite \eqref{EQ:convergence} to get
\begin{equation}
\sum\limits_{n\in\NN}n
\left|
{\hbar}
a_n
\right|.
\end{equation}
By repeating the relevant lines of proof in \cite[page 251]{IT1955}, we can show that the sufficient condition is as follows
$$
\sum\limits_{n\in\NN}n\log(n)|a_n|.
$$
Indeed, by \eqref{cosi}, we have
\begin{multline*}
\sum\limits_{n\in\NN}n
\left|
{\hbar}
a_n
\right|
=
\sum\limits_{n\in\NN}
n
\left|
\sum\limits_{\substack{k\in\NN \\ k-n=\text{odd}}}
a_k
\left(
\frac1{k+n}+\frac1{n-k}
\right)
\right|
\\=
\sum\limits_{n\in\NN}
n
\left|
\sum\limits_{\substack{k\in\NN \\ k-n=\text{odd}}}
a_k
\left(
\frac1{k+n}+\frac1{n-k}
-
\frac2n
\right)
\right|
\\=
\sum\limits_{n\in\NN}
n
\left|
\sum\limits_{\substack{k\in\NN \\ k-n=\text{odd}}}
\frac{k^2a_k}{n(n+k)(n-k)}
\right|
\\=
\sum\limits_{n\in\NN}
n
\left(
\sum\limits_{\substack{k=\overline{1,\ldots, n-1} \\ k-n=\text{odd}}}
\frac{k^2|a_k|}{n(n+k)(n-k)}
+
\sum\limits_{\substack{k=\overline{n,\ldots, \infty} \\ k-n=\text{odd}}}
\frac{k^2|a_k|}{n(n+k)(k-n)}
\right)
\\=
M+N,
\end{multline*}
where we used the fact that $f(0)=f(\pi)$ implies that
$$
\sum\limits_{n\in\NN}a_n=\sum\limits_{n\in\NN}(-1)^na_n=0.
$$
Then we have
\begin{align*}
M
&=
\sum\limits^{n-1}_{k=1}
k^2|a_k|
\sum\limits^{\infty}_{n=k}
\frac1{(n+k)(n-k)}\\
&\leq
\sum\limits^{n-1}_{k=1}
k^2|a_k|
\left(
\sum\limits^{2k}_{n=k}
\frac1{n(n-k)}
+
\sum\limits^{\infty}_{n=2k+1}
\frac2{n^2}
\right)
\\ &\leq
C
\sum\limits^{n-1}_{k=1}
k|a_k|\log(k).
\end{align*}
And for $N$ analogously, we get
\begin{multline*}
N
=
\sum\limits^{\infty}_{k=1}
k^2|a_k|
\sum\limits^{k}_{n=1}
\frac1{(n+k)(n-k)}
\leq
\sum\limits^{\infty}_{k=1}
k^2|a_k|
\left(
\sum\limits^{[\frac{k}2]}_{n=1}
\frac1{k-n}
+
\sum\limits^{k-1}_{n=[\frac{k}2]+1}
\frac1{k-n}
\right)
\\\leq
C
\sum\limits^{\infty}_{k=1}
k|a_k|(\log k + \log k).
\end{multline*}
This completes the proof.
\end{proof}

\begin{proof}[Proof of Theorem \ref{THM:Q1}]
Every Lie group homomorphism gives rise to a Lie algebra homomorphism. The converse is true since $G$ is simply connected. In particular, for every $\pi\in\Gh$, we have
\begin{equation}
\label{EQ:p-dp}
\pi(e^{iH})
=
e^{id\pi(H)},
\end{equation}
where $H\in\mathfrak{t}$. Let $\{H_1,\ldots,H_{l}\}$ in $H$
\begin{equation}
\label{EQ:coordinates}
\mathfrak{t}\ni H \longleftrightarrow (t_1,\ldots,t_l) \colon H=\sum\limits^l_{k=1}t_k H_k.
\end{equation}
By \eqref{EQ:p-dp} and \eqref{EQ:coordinates}, we have
\begin{equation}
\label{EQ:pi-dpi}
\pi(e^{i\sum\limits^l_{k=1}t_kH_k})
=
e^{i\sum\limits^l_{k=1}t_kd\pi(H_k)}.
\end{equation}
It can be proven that matrices $d\pi(H_k)$ can be diagonalised in the representation space of each $\pi\in\Gh$
\begin{equation}
\label{EQ:dpi-diagonal}
d\pi(H_k)
=
\diag(\mu_1,\ldots,\mu_{s})
\end{equation}
with the same $\{\mu_1,\ldots,\mu_s\}\subset \ZZ^l$ from the weight diagram.
We have
\begin{equation}
\sum\limits^l_{k=1}t_k d\pi(H_k)
=
\left(
\begin{matrix}
\sum\limits^l_{k=1}t_k \mu^k_1 & & \\
& \ddots& \\
& & \sum\limits^l_{k=1}t_k \mu^k_s
\end{matrix}
\right).
\end{equation}
By definition, we have
\begin{equation}
\widehat{f}(\pi)
=
\int_{G}f(u)\pi(u)^*\,du.
\end{equation}
Since $f$ and $\pi(u)_{mm}$ are central functions for $m=1,\ldots,\dpi$, the application of Weyl's integral formula \eqref{EQ:Weyl-integral} yields
\begin{equation}
\label{EQ:Fourier-coeff-mm}
\widehat{f}(\pi)_{mm}
=
\frac1{|W|}
\int\limits_{\TT^l}f(t)\overline{\pi(t)_{mm}}\Delta^2(t)\,dt.
\end{equation}
By \eqref{EQ:dpi-diagonal} and \eqref{EQ:pi-dpi}, we get
\begin{equation}
\label{EQ:diagonal-element-of-pi}
\pi(t)_{mm}
=e^{i\mu_m \cdot t}.
\end{equation}
It can be easily shown that
\begin{equation}
\label{EQ:Weyl-denominator}
\Delta(t)
=
\prod\limits_{\alpha\in R_+}
(e^{i(\alpha,H)}+e^{-i(\alpha,H)}-2),
\end{equation}
where $R_+$ are the positive roots associated with $G$.
Thus, using expansion \eqref{EQ:torus-expansion},\eqref{EQ:diagonal-element-of-pi} and \eqref{EQ:Weyl-denominator}, we obtain
\begin{equation}
\widehat{f}(\pi)_{mm}
=
\int\limits_{\TT^l}\sum\limits_{k\in\ZZ^l}a_ke^{2\pi i(k,t)}e^{-2\pi i (\mu_m, t)}\Delta^2(t)\,dt.
\end{equation}
Since the Weyl group $W$ is finite, the last sum can be represented as follows
\begin{equation}
\label{EQ:finite-sum}
\widehat{f}(\pi)_{mm}
=
\sum\limits^{\nu}_{k=1}
a_k \widehat{\Delta^2}(k-\pi_{mm}).
\end{equation}
By the assumption, $f$ is even function (see Definition \ref{DEF:even-odd-G}). Then its Fourier coefficients 
$\widehat{f}(\pi)$ are self-adjoint operators (Proposition \ref{PROP:sa-even})
$$
\widehat{f}(\pi)^*=\widehat{f}(\pi).
$$
Therefore, we have
\begin{equation}
\|\widehat{f}(\pi)\|_{S^1(\mathcal{H}^{\pi})}
=
\sum\limits^{\dpi}_{m=1}|\widehat{f}(\pi)_{mm}|.
\end{equation}
By definition, we have
\begin{equation}
\label{EQ:connection}
\|\widehat{f}\|_{\ell^1_{sch}(\Gh)}
=
\sum\limits_{\pi\in\Gh}
\dpi\|\widehat{f}(\pi)\|_{S^1(\mathcal{H}^{\pi})}
=
\sum\limits_{\pi\in\Gh}
\dpi
\sum\limits^{\dpi}_{m=1}
\left|
\sum\limits^v_{j=1}a_{j}\widehat{\Delta^2}(\pi_{mm}-j)
\right|.
\end{equation}
This completes the proof.
%
\end{proof}


\end{document}